\documentclass[a4paper,12pt]{article}
\usepackage{amssymb}
\usepackage{mathrsfs}
\usepackage{amsmath}
\usepackage{amsfonts,amsthm,amssymb}
\usepackage{amsfonts}
\usepackage{graphics}
\usepackage{color}
\usepackage{graphicx}
\usepackage{epstopdf}
\usepackage{subfigure}
\usepackage{appendix}

\usepackage{geometry}
\newtheorem{lemma}{Lemma}[section]
 \newtheorem{theorem}{Theorem}[section]
\newtheorem{corollary}{Corollary}[section]

\newtheorem{definition}{Definition}[section]

\begin{document}
\title{Extremal Polygonal Cacti for General Sombor Index\footnote{This paper  was supported by the National Natural Science Foundation of China [No.  11971406]. }}
\author{\small Jiachang Ye, Jianguo Qian\footnote{Corresponding author. E-mail: jgqian@xmu.edu (J.G. Qian)}
 \\  \small  School of Mathematical Sciences, Xiamen University,\\ \small Xiamen, 361005, P.R. China}
\date{} \maketitle

\begin{abstract}
The Sombor index of a graph $G$ was recently introduced by Gutman from the geometric point of view, defined as $SO(G)=\sum_{uv\in E(G)}\sqrt{d(u)^2+d(v)^2}$,
 where $d(u)$ is the degree of a vertex $u$.  For two real numbers  $\alpha$ and $\beta$, the $\alpha$-Sombor index and general Sombor index of $G$ are two generalized forms of the Sombor index defined as  $SO_\alpha(G)=\sum_{uv\in E(G)}(d(u)^{\alpha}+d(v)^{\alpha})^{1/\alpha}$ and $SO_\alpha(G;\beta)=\sum_{uv\in E(G)}(d(u)^{\alpha}+d(v)^{\alpha})^{\beta}$, respectively.
  A $k$-polygonal cactus is a connected graph in which every block is a cycle of length $k$. In this paper, we  establish a lower bound on $\alpha$-Sombor index for $k$-polygonal cacti  and show that the bound is attained only by chemical $k$-polygonal cacti. The extremal  $k$-polygonal cacti for $SO_\alpha(G;\beta)$ with  some particular $\alpha$ and $\beta$ are also considered.
\end{abstract}


 \baselineskip=0.30in%行距要求1.5, 老师这里我不太确定是不是这个命令

\section{ Introduction}
We consider only connected simple graphs. For a graph $G$, we denote by $V(G)$ and $E(G)$ the vertex set and edge set of $G$, respectively. For a vertex $v\in V(G)$, we denote by $d_G(v)$, or $d(v)$ if no confusion can occur, the degree of $v$. A vertex $v$ is called a {\em cut vertex} of $G$ if  $G-v$ is not connected.
%For a subset $W\subseteq V(G)$, we denote by $G[W]$ the subgraph of $G$ induced by $W$.

 In mathematical chemistry, particularly in QSPR/QSAR investigation, a large number of topological indices were introduced in an attempt to characterize the physical-chemical properties of molecules. Among these indices, the vertex-degree-based indices play important roles \cite{De10,Gu2013,Gutman}. Probably the most studied are, for examples, the Randi\'{c} connectivity index $R(G)$  \cite{Ran1}, the first and second Zagreb indices $M_1(G)$ and $M_2(G)$ \cite{Gu1972}, which were introduced for the total $\pi$-energy of alternant hydrocarbons.

A vertex-degree-based index of a graph $G$ can be generally represented as the sum of a real function $f(d(u),d(v))$ associated with the edges of $G$ \cite{Gu2020}, i.e.,
$$I_{f}(G)=\sum\limits_{uv\in E(G)} f(d(u), d(v)),$$
where $f(s,t)=f(t,s)$.  In the literature, $I_{f}(G)$ is also called the  {\em connectivity function} \cite{Wang2014} or {\em   bond incident degree index} \cite{Ali,Vu,YJC2}.

Recently, Gutman \cite{Gu2020} introduced an idea to view an edge $e=uv$ as a geometric point, namely the degree-point, that is, to view the ordered pair $(d(u),d(v))$ as the coordinate of $e$.  Therefore, it is interesting to consider the function $f(s,t)$  from the geometric point of view. A natural considering is to define $f(s,t)$  as a geometric distance from the degree-point $(s,t)$ to the origin. In this sense, the first Zagreb index, i.e., $M_1(G)=\sum_{uv\in E(G)}(d(u)+d(v))=\sum_{uv\in E(G)}(|d(u)|+|d(v)|)$, is exactly the index defined on  the Manhattan distance. Along this direction, a more natural considering would be to define $f(s,t)$  as the Euclidean distance, i.e., $f(s,t)=\sqrt{s^2+t^2}$. Indeed, based on this idea, Gutman  \cite{Gu2020} introduced the {\it Somber index} defined by
$$SO(G)=\sum_{uv\in E(G)}\sqrt{d(u)^2+d(v)^2}$$
and further determined the extremal trees for the index. In \cite{Das}, Das et al.  established some  bounds on the Sombor index and some relations between Sombor index and the Zagreb indices and, in  \cite{Red}, Red\v{z}povi\'{c} studied chemical applicability of the Sombor index. Further, Cruz et al. \cite{Cru} characterized the extremal  chemical graphs and hexagonal systems for the Sombor index.

More recently, for positive real number $\alpha$, R\'eti et al. \cite{Reti} defined the {\em $\alpha$-Sombor index} as
$$SO_{\alpha}(G)=\sum_{uv\in E(G)}(d(u)^{\alpha}+d(v)^{\alpha})^{1/\alpha},$$
which could be viewed as the one based on Minkowski distance. In the same paper,  they also characterized the extremal graphs with few cycles for $\alpha$-Sombor index.

In this paper we consider a more generalized form of Sombor index defined as
  $$SO_\alpha(G;\beta)=\sum_{uv\in E(G)}(d(u)^{\alpha}+d(v)^{\alpha})^{\beta},$$
   where $\alpha, \beta$ are real numbers.  We note that this form is a natural generalization of the Sombor index, which was also introduced elsewhere, e.g., the first  $(\alpha,\beta)-KA$ index in \cite{Kulli2019} and the general Sombor index in \cite{Her}. In addition to the first Zagreb, Sombor and the $\alpha$-Sombor index listed above, the general Sombor index also includes many other known indices, e.g., the modified first Zagreb index ($\alpha=-3, \beta=1$) \cite{Nik}, forgotten index ($\alpha=2, \beta=1$) \cite{Fur}, inverse degree index ($\alpha=-2, \beta=1$) \cite{Faj}, modified Sombor index ($\alpha=2, \beta=-1/2$) \cite{Kulli2021}, first Banhatti-Sombor index ($\alpha=-2, \beta=1/2$) \cite{Lin} and general sum-connectivity index ($\alpha=1 , \beta\in \mathbb{R}$) \cite{Zhou1}.
%\vspace*{3mm}

A {\em block} in a graph is a cut edge or a maximal 2-connected component. A {\em cactus} is a connected graph in which every block is a cut edge or a cycle. Equivalently, a cactus has no edge lies in more than one cycle.  In the following, we call a $k$-cycle (a cycle of length $k$) a $k$-{\em polygon}. If each  block of a cactus $G$ is a  $k$-polygon, then $G$ is called a {\em $k$-polygonal cactus} or {\em polygonal cactus} with no confusion.

In this paper,   we consider the extremal $k$-polygonal cacti for $SO_\alpha(G;\beta)$. In the following section we  establish a lower bound on $\alpha$-Sombor index for $k$-polygonal cacti  and show that the bound is attained only by chemical  $k$-polygonal cacti. In the third section we characterize the extremal polygonal cactus with maximum $SO_\alpha(G;\beta)$ for (i)  $\alpha> 1$ and $\beta \geq 1$; and (ii) $1/2\leq \alpha<1$ and $\beta= 2$, respectively. In the fourth section,  we characterize the extremal polygonal cacti with minimum $SO_\alpha(G;\beta)$ for $\alpha>1$ and $\beta \geq 1$.

 \section{Polygonal cacti with minimum $\alpha$-Sombor index}
  For convenience, in what follows we denote $r_\alpha(s,t;\beta)=(s^\alpha+t^\alpha)^\beta$,
   $r_\alpha(s,t)=(s^\alpha+t^\alpha)^{1/\alpha}$ and $r(s,t)=\sqrt{s^2+t^2}$, where $s>0$, $t>0$, $\alpha, \beta \in \mathbb{R}$ and $\alpha \not=0$. For integers $n$ with $n\geq 1$ and $k$ with $k\geq 3$, we denote by $\mathcal {G}_{n,k}$ the class of $k$-polygonal cacti  with  $n$  polygons.

In this section we consider the $\alpha$-Sombor index $SO_{\alpha}(G)$, i.e., $SO_{\alpha}(G;1/ \alpha)$. For $G\in\mathcal{G}_{n,k}$, it is clear that $|V(G)|=nk-n+1,|E(G)|=nk$ and every vertex of $G$ has even degree no more than $2n$. Further, it is clear that $v$ is a cut vertex of $G$ if and only if $v$ has degree no less than 4, i.e., $d_G(v)\geq 4$. A polygon is called a {\em pendent polygon} if it contains exactly one cut-vertex of $G$.  For $2\leq s\leq t\leq 2n$, we denote by $n_{s,t}=n_{s,t}(G)$ the number of edges in $G$ that join two vertices  of degrees $s$ and $t$. Let $X=\{(s,t): s,t\in\{2,4,\ldots,2n\},s\leq t\}$ and $Y=X\setminus\{(2,2), (2,4), (4,4) \}.$

\begin{definition} \label{15.1d}\emph{\cite{Mar1979}}
Let
$\pi=\big(w_{1},w_{2},\ldots,w_{n}\big)$ and
$\pi'=\big(w'_{1},w'_{2},\ldots,w'_{n}\big)$ be two non-increasing  sequences of nonnegative real numbers. We write $\pi\lhd \pi'$
if and only if $\pi\neq  \pi'$,  $\sum_{i=1}^{n}w_{i}=\sum_{i=1}^{n}w'_{i}$, and
$\sum_{i=1}^{j}w_{i}\leq\sum_{i=1}^{j}w'_{i}$ for all
$j=1,\,2,\,\ldots,\,n$.
\end{definition}
A function $\zeta(x)$ defined on a convex set $X$ is called {\em strictly  convex} if
\begin{equation}\label{15.1e}\zeta\big(\mu x_1+(1-\mu)x_2\big)<\mu \zeta(x_1)+(1-\mu)\zeta(x_2)\end{equation}
for  any $0<\mu<1$ and $x_1,x_2\in X$ with $x_1\neq x_2$.
\begin{lemma}\label{15.1l}
\emph{\cite{Mar1979}}
Let  $\pi=\big(w_{1},w_{2},\ldots,w_{n}\big)$ and
$\pi'=\big(w'_{1},w'_{2},\ldots,w'_{n}\big)$ be two non-increasing  sequences of nonnegative real numbers. If
$\pi\lhd \pi'$, then for any strictly convex function $\zeta(x)$, we have  $\sum_{i=1}^{n} \zeta{(w_{i})}<\sum_{i=1}^{n} \zeta{(w'_{i})}$.
\end{lemma}
\begin{lemma}\label{12.6l}
Let $\alpha>1$ and $n\geq 3$. Then\\
(i). $r_\alpha(2n,2)-r_\alpha(2n-2,4)>0;$\\
(ii). $r_\alpha(6,2)+r_\alpha(2,2)-2r_\alpha(4,2)>0.$
\end{lemma}
\begin{proof}
Since $\alpha>1$ and $n\geq 3$, then by Lemma \ref{15.1l}, we have $(2n)^{\alpha}+2^{\alpha}>(2n-2)^{\alpha}+4^{\alpha}$. Hence (i) holds clearly.
Let $g(x)=r_\alpha(x,2)=(x^\alpha+2^\alpha)^{1/\alpha}$, where $x>0$ and $\alpha>1$. Since $g^{''}(x)=\frac{(\alpha-1)2^{\alpha}x^{\alpha-2}}{(x^\alpha+2^\alpha)^{2-1/\alpha}}>0$, then $g(x)$ is strictly  convex. Then by Lemma \ref{15.1l}, $g(6)+g(2)>2g(4)$. Hence (ii) also holds.
\end{proof}
\begin{lemma}\label{13.2l}  Let $\alpha \neq0$ and $G\in\mathcal{G}_{n,k}$, where $n\geq 1$ and $k\geq 3$. Then
$$ SO_{\alpha}(G)=(4n-4)(2^\alpha+4^\alpha)^{1/\alpha}+2(nk-4n+4)2^{1/\alpha}$$
$$~~~~~~+\left(6\times2^{1/\alpha}-2(2^\alpha+4^\alpha)^{1/\alpha}\right)n_{4,4}
+\sum_{(s,t)\in Y}\eta(s,t;\alpha)n_{s,t},$$
 where $\eta(s,t;\alpha)=(s^{\alpha}+t^{\alpha})^{1/\alpha}-2\left(\frac{1}{s}+\frac{1}{t}\right)2^{1/\alpha}$.
\end{lemma}
\begin{proof}  By the definition of $n_{s,t}$, it is not difficult to see that
\begin{equation}\label{13.3e}
\left\{\begin{array}{ccl}
nk-n+1&=&\sum\limits_{(s,t)\in X}\left(\frac{1}{s}+\frac{1}{t}\right)n_{s,t},\\
nk&=&\sum\limits_{(s,t)\in X}n_{s,t}
\end{array}
\right.
\end{equation}
because each vertex of $G$ contributes 1 to each of the two sides in the first equation and  each edge of $G$ contributes 1 to each of the two sides in the second equation.  Write (\ref{13.3e}) as
\begin{equation}
\left\{\begin{array}{ccl}
4n_{2,2}+3n_{2,4}&=&4(nk-n+1)-2n_{4,4}-4\sum\limits_{(s,t)\in Y}\left(\frac{1}{s}+\frac{1}{t}\right)n_{s,t},\\
n_{2,2}+n_{2,4}&=&nk-n_{4,4}-\sum\limits_{(s,t)\in Y}\left(\frac{1}{s}+\frac{1}{t}\right)n_{s,t}.
\end{array}
\right.
\end{equation}
Therefore,
\begin{equation}\label{13.5e}
\left\{\begin{array}{ccl}
n_{2,4}&=&4n-4-2n_{4,4},\\
n_{2,2}&=&nk-4n+4+n_{4,4}-\sum\limits_{(s,t)\in Y}\left(\frac{1}{s}+\frac{1}{t}\right)n_{s,t}.
\end{array}
\right.
\end{equation}

Consequently, by (\ref{13.5e}) we have
\begin{eqnarray*}
SO_{\alpha}(G)&= &(4^\alpha+4^\alpha)^{1/\alpha}n_{4,4}+(2^\alpha+4^\alpha)^{1/\alpha}n_{2,4}
+(2^\alpha+2^\alpha)^{1/\alpha}n_{2,2}\\
 &~~+
 &\sum_{(s,t)\in Y}(s^\alpha+t^\alpha)^{1/\alpha}n_{s,t}\\
 &=
 &(4n-4)(2^\alpha+4^\alpha)^{1/\alpha}+2(nk-4n+4)2^{1/\alpha}\\
 &~~+
 &\left(6\times2^{1/\alpha}-2(2^\alpha+4^\alpha)^{1/\alpha}\right)n_{4,4}\\
 &~~+
 &\sum_{(s,t)\in Y}\left((s^{\alpha}+t^{\alpha})^{1/\alpha}-2\left(\frac{1}{s}+\frac{1}{t}\right)2^{1/\alpha}\right)n_{s,t}.
\end{eqnarray*}
\end{proof}

 For $\alpha\neq0$ and positive integer $p$, let  $\delta_{\alpha,p}(s,t)=\left((s+p)^{\alpha}+t^{\alpha}\right)^{1/\alpha}-\left(s^{\alpha}+t^{\alpha}\right)^{1/\alpha}$, where $s, t>0$.
\begin{lemma}\label{13.3l} If $s, t>0$ and $p$ is an arbitrary positive integer, then\\
(i). $r_{\alpha}(s,t;\beta)$ strictly increases in $s$ for fixed $t$, and in $t$ for fixed $s$ when $\alpha, \beta>0 $;\\
(ii). $\delta_{\alpha,p}(s,t)>0$ and $\delta_{\alpha,p}(s,t)$ strictly decreases in $t$ for fixed $s$ when $\alpha>1$;\\
(iii). $\delta_{\alpha,p}(s,t)$ strictly increases in $s$ for fixed $t$ when $\alpha>1$.
\end{lemma}
\begin{proof}  (i) follows directly since $\alpha, \beta>0$.

Since $-1<1/\alpha-1<0$ when $\alpha>1$,
$$\frac{\partial \delta_{\alpha,1}(s,t)}{\partial t}=t^{\alpha-1}\big(\left((s+1)^{\alpha}+t^{\alpha}\right)^{1/\alpha-1}-\left(s^{\alpha}+t^{\alpha}\right)^{1/\alpha-1}\big)<0.$$
We also note that $\delta_{\alpha,1}(s,t)>0$, hence (ii) follows as $\delta_{\alpha,p}(s,t)=\sum^{p-1}_{i=0}\delta_{\alpha,1}(s+i,t)$.

Finally,  since $1-1/\alpha>0$ when $\alpha>1$,
$$\frac{\partial \delta_{\alpha,1}(s,t)}{\partial s}=\frac{\big((s+1)^{\alpha}s^{\alpha}+(s+1)^{\alpha}t^{\alpha}\big)^{1-1/\alpha}-
\big((s+1)^{\alpha}s^{\alpha}+s^{\alpha}t^{\alpha}\big)^{1-1/\alpha} }
{\big(((s+1)^{\alpha}+t^{\alpha})(s^{\alpha}+t^{\alpha})\big)^{1-1/\alpha} }>0.$$
Hence (iii) follows as $\delta_{\alpha,p}(s,t)=\sum^{p-1}_{i=0}\delta_{\alpha,1}(s+i,t)$.
\end{proof}

 The {\em distance} $d_G(u, v)$ between two vertices $u$ and $v$ of a connected  graph $G$ is defined as usual as the
length of a shortest path that connects $u$ and $v$. In general, for two subgraphs $G_1$ and $G_2$ of $G$, we define the distance  between $G_1$ and $G_2$ by $d_G(G_1,G_2)=\min\{d_G(u,v):u\in V(G_1),v\in V(G_2)\}$. For $n\geq 2$, a {\em star-like cactus} $S_{n,k}$ is defined intuitively as a $k$-polygonal cactus such that all polygons have a vertex in common. It is clear that $S_{n,k}$ is unique  and contains  exactly one vertex of  degree  $2n$  while all other vertices have degree two.

\begin{lemma}\label{13.4l}Let $\alpha>1$ and $G\in \mathcal{G}_{n,k}$, where  $n\geq 3$ and  $k\geq3$. If $G$ contains a vertex of degree at least 6, then $SO_\alpha(G)$ is not minimum in $\mathcal{G}_{n,k}$.
\end{lemma}
\begin{proof} Suppose to the contrary that $G$ contains $q$ ($q\geq 1$) vertices of degree at least 6 and $SO_\alpha(G)$ is  minimum in $\mathcal{G}_{n,k}$.\\
{\bf Case 1. $q\geq 2$.}

Let $u_1$ and  $w$ be two  vertices of degree at least 6 such that $d_G(u_1,w)$ is maximum and let $P$ be a shortest path connecting $u_1$ and $w$. Since $d_G(u_1)\geq 6$, $u_1$ is contained in at least three polygons, exactly one of which, say $C$, has at least two common vertices with $P$. Let $C_1=u_1u_2 \cdots u_ku_1$ and $C_2=u_1z_2 \cdots z_ku_1$  be two  polygons other than $C$ that contain $u_1$  as a common vertex. Since $d_G(u_1,w)$ is maximum, we have $d_G(v)\leq 4$ for every $v\in\{u_2, u_k, z_2,z_k\}$. Let $C_3=v_1v_2 \cdots v_kv_1$ be a pendent polygon that lies in the same component with $w$ in $G-u_1$ and the distance $d_G(u_1,C_3)$ is as large as possible, where $d_G(v_1)=2a\geq 4$ and $d_G(v_2)=d_G(v_3)=2$.

Without loss of generality, assume $d_G(u_1)=2b\geq d_G(w)\geq 6$. Let $G'=G-u_1u_2-u_1u_k+v_2u_2+v_2u_k$. Then by Lemma \ref{13.3l}, we have

\begin{eqnarray*}
SO_\alpha(G)-SO_\alpha(G')&>
  &\big(r_\alpha(2b,d(u_2))- r_\alpha(4,d(u_2)\big)+
  \big(r_\alpha(2b,d(u_k))- r_\alpha(4,d(u_k)\big)\\
  &&+\big(r_\alpha(2,2)- r_\alpha(4,2)\big)+\big(r_\alpha(2,2a)- r_\alpha(4,2a)\big)\\
  &&+\big(r_\alpha(2b,d(z_2))- r_\alpha(2b-2,d(z_2))\big)\\
  &&+\big(r_\alpha(2b,d(z_k))- r_\alpha(2b-2,d(z_k))\big)\\
  &>
  &\big(r_\alpha(6,4)- r_\alpha(4,4)\big)+\big(r_\alpha(6,4)- r_\alpha(4,4)\big)\\
  &&+\big(r_\alpha(2,2)- r_\alpha(4,2)\big)+\big(r_\alpha(2,4)- r_\alpha(4,4)\big)\\
  &&+\big(r_\alpha(6,4)- r_\alpha(4,4)\big)+\big(r_\alpha(6,4)- r_\alpha(4,4)\big)\\
  &=
  &8(3^{\alpha}+2^{\alpha})^{1/\alpha}-18\times2^{1/\alpha}\\
  &>
  &8\sqrt{6}\times2^{1/\alpha}-18\times2^{1/\alpha}\\
  &>
  &0,
\end{eqnarray*}
which contradicts the minimality  of $G$.

\noindent
{\bf Case 2. $q= 1$.}

If $G \not \cong S_{n,k}$, then the discussion for this case is similar to that for Case 1 by choosing $u_1$ to be the vertex with degree at least 6 and $C_3=v_1v_2 \cdots v_kv_1$ to be a pendent polygon such that $d_G(u_1,C_3)$ is maximum. Otherwise, $G \cong S_{n,k}$. Let $u_1$  be the vertex with degree $2n\geq 6$, $C_1=u_1u_2 \cdots u_ku_1$ and $C_2=u_1z_2 \cdots z_ku_1$  be two  pendent polygons. Let $G'=G-u_1u_2-u_1u_k+z_2u_2+z_2u_k$. Then by Lemma \ref{13.3l} and
Lemma \ref{12.6l}, we have
\begin{eqnarray*}
SO_\alpha(G)-SO_\alpha(G')&=
  &\big( 2n\times r_\alpha(2n,2)+(nk-2n)\times r_\alpha(2,2) \big)\\
  &&-\big( (2n-3)\times r_\alpha(2n-2,2)+r_\alpha(2n-2,4) \\
  &&+3r_\alpha(4,2)+(nk-2n-1)r_\alpha(2,2) \big)\\
  &=
  &2n\times r_\alpha(2n,2)+r_\alpha(2,2)- (2n-3)\times r_\alpha(2n-2,2)\\
  &&-r_\alpha(2n-2,4)-3r_\alpha(4,2) \\
  &>
  &3r_\alpha(2n,2)+r_\alpha(2,2)-r_\alpha(2n-2,4)-3r_\alpha(4,2)\\
  &>
  &2r_\alpha(2n,2)+r_\alpha(2,2)-3r_\alpha(4,2)\\
  &>
  &r_\alpha(6,2)+r_\alpha(2,2)-2r_\alpha(4,2)\\
  &>
  &0,
\end{eqnarray*}
which contradicts the minimality  of $G$.

\end{proof}

Recall that a graph is called a chemical graph if it has no vertex of degree more than 4. For $G\in\mathcal{G}_{n,k}$, we call $G$ a {\it chemical $(n,k)$-cactus}, or {\it chemical cactus} for short, if $G$ has no vertex of degree greater than 4. It is clear that every cut vertex in a chemical cactus has degree 4, which connects exactly two polygons.  The following corollary follows directly from Lemma \ref{13.2l} and Lemma \ref{13.4l}, which shows that the minimum value of  $SO_\alpha(G)$ among all cacti in $\mathcal{G}_{n,k}$ is attained only by  chemical cacti.
 \begin{corollary}\label{13.1c} For $\alpha>1$, $n\geq 3,k\geq 3$ and $G\in \mathcal{G}_{n,k}$, if $G$ attains the minimum value of $SO_\alpha(G)$, then $G$ is a chemical cactus and
 $$ SO_{\alpha}(G)=(4n-4)(2^\alpha+4^\alpha)^{1/\alpha}+2(nk-4n+4)2^{1/\alpha}
+\left(6\times2^{1/\alpha}-2(2^\alpha+4^\alpha)^{1/\alpha}\right)n_{4,4}(G).$$
 \end{corollary}

In the following we will determine the minimum value of  $SO_{\alpha}(G)$ among all chemical cacti. By Corollary \ref{13.1c}, this is equivalent to determine the maximum value of $n_{4,4}(G)$ as $6\times2^{1/\alpha}-2(2^\alpha+4^\alpha)^{1/\alpha}<6\times2^{1/\alpha}-2(2\times3^\alpha)^{1/\alpha}=0$ by Lemma \ref{15.1l}. For a chemical cactus $H$, we call a polygon $C$ in $H$ a {\it saturated polygon} if every vertex on $C$ is a cut vertex, i.e., a vertex of degree 4. Further, we call a chemical cactus $H$ {\it nice-saturated} if the following two conditions hold:\\
1). $H$ has as many as possible  saturated polygons;\\
2). the cut vertices on each polygon of $H$ are successively arranged.

For a chemical cactus $H$, let $T(H)$ be the tree whose vertices are the polygons in $H$ and two vertices are adjacent provided their corresponding polygons has a common vertex. It is clear that $T(H)$ is a tree with maximum vertex degree no more than $k$. Let $p$ be the number of the vertices of degree $k$ in $T(H)$, and let $d_1,d_2,\ldots,d_s$ be the degrees of all the vertices in $H$ that are neither of degree 1 nor of degree $k$, i.e., $1<d_i<k$ for each $i\in\{1,2,\ldots,s\}$. Since $T(H)$ is a tree, we have
\begin{equation}\label{e6}
kp+d_1+d_2+\cdots+d_s+(n-p-s)=2n-2
\end{equation}
 and every saturated polygon in $H$ corresponds to a vertex of degree $k$ in $T(H)$. Further, $T(H)$ has as many as possible vertices of degree $k$ if and only if $d_1+d_2+\cdots+d_s-s<k-1$. This implies that
\begin{equation}\label{p}
\frac{n-2}{k-1}-1 < p= \frac{n-2-(d_1+d_2+\cdots+d_s-s)}{k-1}\leq \frac{n-2}{k-1}.
\end{equation}
That is, if $H$ is nice-saturated then $H$ has exactly $\left\lfloor\frac{n-2}{k-1}\right\rfloor$ saturated cycles. As an example,
a chemical $(6,4)$-cactus with $p<\left\lfloor\frac{n-2}{k-1}\right\rfloor=1$, a chemical $(6,4)$-cactus in which the cut vertices on some polygon are not successively arranged, and a nice-saturated $(6,4)$-cactus are illustrated as (a), (b) and (c), respectively, in Figure 1.
%老师我的版本编译您新的图片格式会报错，我的要把31.EPS写成31.eps
%\begin{figure}[htbp]\label{Fig3.2}
% \centering
% \includegraphics[height=4.2cm]{31.EPS}
%\caption{(a). $p=0$; (b). The cut vertices on the $k$-cycle $C$ are not successively arranged; (c). A nice-saturated $(6,4)$-cactus. }
%\end{figure}
\begin{figure}[htbp]\label{Fig3.2}
 \centering
 \includegraphics[height=4.2cm]{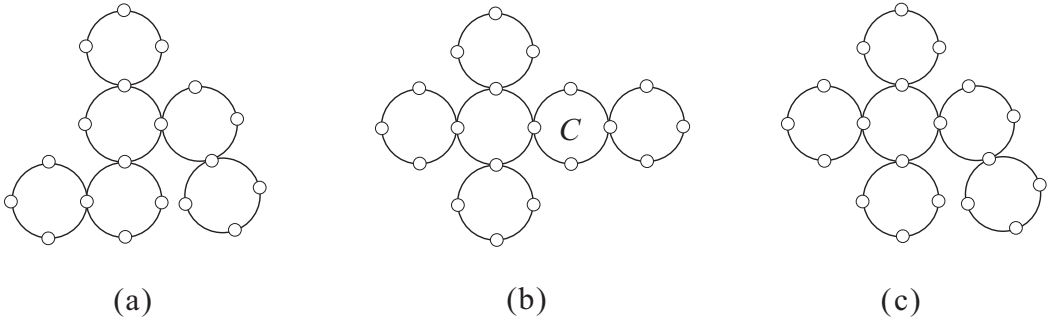}
\caption{(a). $p=0$; (b). The cut vertices on the $k$-cycle $C$ are not successively arranged; (c). A nice-saturated $(6,4)$-cactus. }
\end{figure}
\begin{lemma}\label{13.5l}  A chemical cactus $H$ attains the maximum value of $n_{4,4}(H)$ if and only if $H$ is nice-saturated.
\end{lemma}
\noindent
\begin{proof}  Let $p$ and $d_1,d_2,\ldots,d_s$ be defined as above.  By a simple calculation, we have
\begin{equation}\label{degree}
n_{4,4}(H)\leq kp+\sum_{v\in V(G),d(v)<k}(d(v)-1)=kp+(d_1-1)+(d_2-1)+\cdots+(d_s-1)
\end{equation}
and the equality holds if and only if the cut vertices on each polygon of $H$ are successively arranged.

Suppose $p<\left\lfloor\frac{n-2}{k-1}\right\rfloor$. Then by the pervious analysis, we have $d_1+d_2+\cdots+d_s-s\geq k-1$. Let $d'_1,d'_2,\ldots,d'_s$ be a sequence satisfying $d'_1=k, 1\leq d'_i\leq d_i$ for $i\in\{2,3,\cdots,s\}$ and $\sum_{i=1}^sd'_i=\sum_{i=1}^sd_i$. Let {\bf S} be the sequence obtained from the degree sequence of $T(H)$ by replacing $d_1,d_2,\ldots,d_s$ by  $d'_1,d'_2,\ldots,d'_s$, respectively. It is clear that {\bf S} is still a degree sequence of a tree with maximum degree not greater than $k$. Let $H'$ be a cactus such that $T(H')$ has degree sequence   {\bf S} and  the cut vertices on each polygon of $H'$ are successively arranged. Then by (\ref{degree}) and a direct calculation, we have $n_{4,4}(H')=n_{4,4}(H)+1$. That is, $H$ does not attain the maximum value of $n_{4,4}(H)$, which completes our proof.
\end{proof}
\begin{theorem}\label{13.1t}  Let $G\in \mathcal{G}_{n,k}$, where $n\geq3$ and $k\geq3$. Then
$$SO_{\alpha}(G)\geq (4n-4)(2^\alpha+4^\alpha)^{1/\alpha}+2(nk-4n+4)2^{1/\alpha}$$
$$~~~~~~~~~~+\left(6\times2^{1/\alpha}-2(2^\alpha+4^\alpha)^{1/\alpha}\right) \big( n-2+ \left\lfloor \frac{n-2}{k-1} \right\rfloor \big),$$
and the equality holds if and only if $G$ is a nice-saturated chemical cactus.
\end{theorem}
\begin{proof}  By (\ref{degree}), (\ref{e6})  and (\ref{p}), if $G$ is minimum, then
$$n_{4,4}(G)=kp+(d_1-1)+(d_2-1)+\cdots+(d_s-1)=n-2+p=n-2+ \left\lfloor \frac{n-2}{k-1} \right\rfloor.$$
Hence, the theorem follows directly from Corollary \ref{13.1c} and Lemma \ref{13.5l}.
\end{proof}

\section{Polygonal cactus with maximum general Sombor index}
In this section we will characterize the polygonal cactus with maximum general Sombor index for the two cases  $\alpha \geq 1, \beta> 1$; and  $\alpha= 2, 1/2\leq \beta<1$, respectively.
%
%\begin{definition} \label{15.1d}\emph{\cite{Mar1979}}
%Let
%$\pi=\big(w_{1},w_{2},\ldots,w_{n}\big)$ and
%$\pi'=\big(w'_{1},w'_{2},\ldots,w'_{n}\big)$ be two non-increasing  sequences of nonnegative real numbers. We write $\pi\lhd \pi'$
%if and only if $\pi\neq  \pi'$,  $\sum_{i=1}^{n}w_{i}=\sum_{i=1}^{n}w'_{i}$, and
%$\sum_{i=1}^{j}w_{i}\leq\sum_{i=1}^{j}w'_{i}$ for all
%$j=1,\,2,\,\ldots,\,n$.
%\end{definition}
%A function $\zeta(x)$ defined on a convex set $X$ is called {\em strictly  convex} if
%\begin{equation}\label{15.1e}\zeta\big(\mu x_1+(1-\mu)x_2\big)<\mu \zeta(x_1)+(1-\mu)\zeta(x_2)\end{equation}
%for  any $0<\mu<1$ and $x_1,x_2\in X$ with $x_1\neq x_2$.
%\begin{lemma}\label{15.1l}
%\emph{\cite{Mar1979}}
%Let  $\pi=\big(w_{1},w_{2},\ldots,w_{n}\big)$ and
%$\pi'=\big(w'_{1},w'_{2},\ldots,w'_{n}\big)$ be two non-increasing  sequences of nonnegative real numbers. If
%$\pi\lhd \pi'$, then for any strictly convex function $\zeta(x)$, we have  $\sum_{i=1}^{n} \zeta{(w_{i})}<\sum_{i=1}^{n} \zeta{(w'_{i})}$.
%\end{lemma}
\begin{lemma}\label{12.2l}
Let  $\Delta ABM$ be a triangle in Euclidean space and $O$ the midpoint of the triangle side $AB$. Then $|MA|^{2\beta}+|MB|^{2\beta}>2|MO|^{2\beta}$ for any real number $\beta\geq \frac{1}{2}$, where $|MA|$ is the length of the side $MA$.
\end{lemma}
\begin{proof} Let $|MA|=a$, $|MB|=b$ and $|MO|=d$.  When $a=b$, the lemma follows directly. Without loss of generality, we now assume $a>b>0$.  By the triangle inequality, $d<\frac{a+b}{2}<a$ and so $(a,b)\rhd \left(\frac{a+b}{2},\frac{a+b}{2}\right)$. Hence,
 by Lemma \ref{15.1l},
 $a^{2\beta}+b^{2\beta}\geq 2\left(\frac{a+b}{2}\right)^{2\beta}>2d^{2\beta}$ when $\beta \geq \frac{1}{2}.$
\end{proof}

\begin{lemma}\label{12.4l}
Let $s>2$ and $t>2$. Then\\
(i). $r_\alpha(s+2,2;\beta)-r_\alpha(s-2,2;\beta)>0$ for any $\alpha >0$ and $\beta>0$;\\
(ii). $r_\alpha(s+2,t;\beta)+r_\alpha(s-2,t;\beta)\geq 2r_\alpha(s,t;\beta)$ for any $\alpha \geq 1$  and $\beta>1$;\\
(iii). $r_\alpha(s+2,t-2;\beta)+r_\alpha(s-2,t+2;\beta)\geq 2r_\alpha(s,t;\beta)$ for any  $\alpha \geq 1$ and $\beta>1$.
\end{lemma}
\begin{proof} (i) follows directly.

For (ii), by Lemma \ref{15.1l} and the monotonicity of $r_\alpha(s,t;\beta)$, we have
\begin{eqnarray*}
r_\alpha(s+2,t;\beta)+r_\alpha(s-2,t;\beta) &=& \left((s+2)^{\alpha}+t^{\alpha} \right)^{\beta}+
\left((s-2)^{\alpha}+t^{\alpha} \right)^{\beta}\\
&&\geq 2\left( \frac{(s+2)^{\alpha}+(s-2)^{\alpha}}{2}+t^{\alpha} \right)^{\beta}\\
&&\geq 2(s^{\alpha}+t^{\alpha})^{\beta}\\
&&= 2r_\alpha(s,t;\beta).
\end{eqnarray*} Hence, (ii) holds.

The discussion for (iii) is analogous to that for (ii).
\end{proof}

%For $n\geq 2$, a {\em star-like cactus} $S_{n,k}$ is defined intuitively as a $k$-polygonal cactus such that all polygons have a vertex in common. It is clear that $S_{n,k}$ is unique  and contains  exactly one vertex of  degree  $2n$  while all other vertices have degree two.

%%%%%%%%%%%%%%%%%%%%%%%%%%%%%%%%%%%%%%%%%%第三节改到这

\begin{theorem}\label{12.1t} Let  $n\geq 3,k\geq 3$ and $G\in \mathcal{G}_{n,k}$. If $\alpha \geq 1$ and $\beta> 1$ ; or $\alpha= 2$ and $\frac{1}{2}\leq \beta<1$, then
$$SO_{\alpha}(G;\beta)\leq 2n((2n)^{\alpha}+2^{\alpha})^{\beta}+n(k-2)(2^{\alpha+1})^{\beta}$$
and the equality holds if and only if $G\cong S_{n,k}$.
\end{theorem}
\begin{proof} We first assume that  $\alpha \geq 1$ and $\beta> 1$.

Let $G$ be such that $SO_{\alpha}(G;\beta)$ is as large as possible. Further, let $C_1=z_1z_2\cdots z_kz_1$ and  $C_2=v_1v_2\cdots v_kv_1$ be two  pendent polygons such that $d_{G}(C_1,C_2)$ is as large as possible, where $z_1$ and $v_1$  are the cut-vertices of $C_1$ and $C_2$, respectively.

If $G\cong S_{n,k}$, then the theorem follows directly. We now assume $G\not\cong S_{n,k}$. Then, $z_1\neq v_1$. Let $G_1=G-v_1v_2-v_1v_k+z_1v_2+z_1v_k$ and  $G_2=G-z_1z_2-z_1z_k+v_1z_2+v_1z_k$. We consider the following two cases:

\noindent
{\bf Case 1.} $z_1$ and $v_1$ are adjacent in $G$.

In this case, we have
$SO_\alpha(G_1;\beta)-SO_\alpha(G;\beta)=$
\begin{eqnarray*}
&&\sum_{v\in N_G(v_1)\setminus \{z_1\}}\left(r_\alpha\big(d_G(v_1)-2,d_G(v);\beta\big)-r_\alpha\big(d_G(v_1),d_G(v);\beta\big)\right)\\[2mm]
&&+\sum_{z\in N_G(z_1)\setminus \{v_1\}}\left(r_\alpha\big(d_G(z_1)+2,d_G(z);\beta\big)-r_\alpha\big(d_G(z_1),d_G(z);\beta\big)\right)\\[2mm]
&&+2r_\alpha\big(d_G(z_1)+2,2;\beta\big)-2r_\alpha\big(d_G(v_1)-2,2;\beta\big)\\
&&+r_\alpha\big(d_G(v_1)-2,d_G(z_1)+2;\beta\big)-r_\alpha\big(d_G(v_1),d_G(z_1);\beta\big),\,\,\text{and}
\end{eqnarray*}

$SO_\alpha(G_2;\beta)-SO_\alpha(G;\beta)=$
\begin{eqnarray*}
&&\sum_{v\in N_G(v_1)\setminus \{z_1\}}\left(r_\alpha\big(d_G(v_1)+2,d_G(v);\beta\big)-r_\alpha\big(d_G(v_1),d_G(v);\beta\big)\right)\\[2mm]
&&+\sum_{z\in N_G(z_1)\setminus \{v_1\}}\left(r_\alpha\big(d_G(z_1)-2,d_G(w);\beta\big)-r_\alpha\big(d_G(z_1),d_G(z);\beta\big)\right)\\[2mm]
&&+2r_\alpha\big(d_G(v_1)+2,2;\beta\big)-2r_\alpha\big(d_G(z_1)-2,2;\beta\big)\\
&&+r_\alpha\big(d_G(v_1)+2,d_G(z_1)-2;\beta\big)-r_\alpha\big(d_G(v_1),d_G(z_1);\beta\big).
\end{eqnarray*}

Recall that $z_1$ and $v_1$  are the cut-vertices of  $C_1$ and $C_2$, respectively.
Therefore, $d_G(z_1)\geq 4$ and $d_G(v_1)\geq 4$.  Combining with Lemma  \ref{12.4l}, we have
\begin{eqnarray*}
&&r_\alpha\big(d_G(v_1)+2,d_G(v);\beta\big)+r_\alpha\big(d_G(v_1)-2,d_G(v);\beta\big)>2r_\alpha\big(d_G(v_1),d_G(v);\beta\big),\\
&&r_\alpha\big(d_G(z_1)+2,d_G(z);\beta\big)+r_\alpha\big(d_G(z_1)-2,d_G(z);\beta\big)>2r_\alpha\big(d_G(z_1),d_G(z);\beta\big),\\
&&r_\alpha\big(d_G(v_1)+2,d_G(z_1)-2;\beta\big)+r_\alpha\big(d_G(v_1)-2,d_G(z_1)+2;\beta\big)> 2r_\alpha\big(d_G(v_1),d_G(z_1);\beta\big),\\
&&r_\alpha\big(d_G(z_1)+2,2;\beta\big)-r_\alpha\big(d_G(z_1)-2,2;\beta\big)>0, \text{\quad and}\\
&& r_\alpha\big(d_G(v_1)+2,2;\beta\big)-r_\alpha\big(d_G(v_1)-2,2;\beta\big)>0.
\end{eqnarray*}
This means that $SO_\alpha(G_1;\beta)> SO_\alpha(G;\beta)$ or $SO_\alpha(G_2;\beta)> SO_\alpha(G;\beta)$, a contradiction.

\noindent
{\bf Case 2.} $z_1$ and $v_1$ are not adjacent in $G$.

In this case,  we have
\begin{eqnarray*}
SO_\alpha(G_1;\beta)-SO_\alpha(G;\beta) &=&\sum_{v\in N_G(v_1)}\left(r_\alpha\big(d_G(v_1)-2,d_G(v);\beta\big)-r_\alpha\big(d_G(v_1),d_G(v);\beta\big)\right)\\[2mm]
&&+\sum_{z\in N_G(z_1)}\left(r_\alpha\big(d_G(z_1)+2,d_G(z);\beta\big)-r_\alpha\big(d_G(z_1),d_G(z);\beta\big)\right)\\[2mm]
&&+2r_\alpha\big(d_G(z_1)+2,2;\beta\big)-2r_\alpha\big(d_G(v_1)-2,2;\beta\big),\,\,\text{and}
\end{eqnarray*}
\begin{eqnarray*}
SO_\alpha(G_2;\beta)-SO_\alpha(G;\beta)&=&\sum_{v\in N_G(v_1)}\left(r_\alpha\big(d_G(v_1)+2,d_G(v);\beta\big)-r_\alpha\big(d_G(v_1),d_G(v);\beta\big)\right)\\[2mm]
&&+\sum_{z\in N_G(z_1)}\left(r_\alpha\big(d_G(z_1)-2,d_G(z);\beta\big)-r_\alpha\big(d_G(z_1),d_G(z);\beta\big)\right)\\[2mm]
&&+2r_\alpha\big(d_G(v_1)+2,2;\beta\big)-2r_\alpha\big(d_G(z_1)-2,2;\beta\big).
\end{eqnarray*}

Recall that $d_G(z_1)\geq 4$ and $d_G(v_1)\geq 4$. Similar to Case 1, by Lemma  \ref{12.4l} , we have
\begin{eqnarray*}
&&r_\alpha\big(d_G(v_1)+2,d_G(v);\beta\big)+r_\alpha\big(d_G(v_1)-2,d_G(v);\beta\big)>2r_\alpha\big(d_G(v_1),d_G(v);\beta\big),\\
&&r_\alpha\big(d_G(z_1)+2,d_G(z);\beta\big)+r_\alpha\big(d_G(z_1)-2,d_G(z);\beta\big)>2r_\alpha\big(d_G(z_1),d_G(z);\beta\big),\\
&&r_\alpha\big(d_G(z_1)+2,2;\beta\big)-r_\alpha\big(d_G(z_1)-2,2;\beta\big)>0  \text{\quad and}\\
&& r_\alpha\big(d_G(v_1)+2,2;\beta\big)-r_\alpha\big(d_G(v_1)-2,2;\beta\big)>0,
\end{eqnarray*}
which means that $SO_\alpha(G_1;\beta)> SO_\alpha(G;\beta)$ or $SO_\alpha(G_2;\beta)> SO_\alpha(G;\beta)$, a contradiction.

 Therefore, $S_{n,k}$ is the unique maximal polygonal cactus. Further, we have
 $$SO_\alpha(S_{n,k};\beta)=2n r_{\alpha}(2n,2;\beta)+n(k-2)r_{\alpha}(2,2;\beta)=2n((2n)^{\alpha}+2^{\alpha})^{\beta}+n(k-2)(2^{\alpha+1})^{\beta}.$$

 The discussion for the case that $\alpha= 2$ and $\frac{1}{2}\leq \beta<1$ is analogous by Lemma \ref{12.4l} (i) and  Lemma \ref{12.2l}.\end{proof}

 \section{Polygonal cacti with minimum general Sombor index}
A symmetric function $\varphi(s,t)$ defined on positive real numbers is called {\em escalating}  \cite{Wang2014} if
\begin{align}\label{e21e}
\varphi(s_{1},s_{2})+\varphi(t_1,t_2)\geq  \varphi(s_2,t_1)+\varphi(s_1,t_2)
\end{align}
for any $s_1\geq t_1>0$  and $s_2\geq t_2>0,$ and the inequality holds if $s_1>t_1>0$ and $s_2>t_2>0$.
Further, an escalating function $\varphi(s,t)$ is called {\em special  escalating} \cite{YJC2} if
\begin{align}\label{e22e}
4\varphi(2l,2)-\varphi(2l-2,4)-\varphi(2l-2,2)-\varphi(4,2)-\varphi(4,4)\geq 0
\end{align}
for $l\geq 3$ and
\begin{align}\label{e23e}
\varphi(s_1,s_2)-\varphi(t_1,t_2)\geq 0
\end{align}
for any $s_1\geq t_1\geq 2$ and $s_2\geq t_2\geq 2$.
\begin{lemma}\label{15.3l} If $s, t>0$, then  $r_\alpha(s,t;\beta)=\left(s^\alpha+t^\alpha\right)^\beta$ is special  escalating for $\alpha\geq 1$ and $\beta>1$.
\end{lemma}
\begin{proof} Set $\varphi(s,t)=\left(s^\alpha+t^\alpha\right)^\beta$. Since $\alpha\geq 1$ and $\beta>1$, we have $\big(s_{1}^{\alpha}+s_{2}^{\alpha},t_{1}^{\alpha}+t_{2}^{\alpha}\big)\rhd \big(s_{2}^{\alpha}+t_{1}^{\alpha},s_1^{\alpha}+t_{2}^{\alpha}\big)$ when $s_1>t_1>0$ and $s_2>t_2>0$. Then by Lemma \ref{15.1l}, the inequality in (\ref{e21e})  strictly holds. Further, it is clear that the equality in (\ref{e21e}) holds when $s_1=t_1>0$ or $s_2=t_2>0$. This means that $\left(s^\alpha+t^\alpha\right)^\beta$ is escalating.

In addition,  by Lemma \ref{15.1l} and the monotonicity of $\left(s^\alpha+t^\alpha\right)^\beta$, if $l\geq 3$ and $\alpha \geq 1$ then $(2l)^{\alpha} +2^{\alpha}\geq (2l-2)^{\alpha} +4^{\alpha}>(2l-2)^{\alpha} +2^{\alpha}$, $(2l)^{\alpha} +2^{\alpha}> 4^{\alpha}+2^{\alpha}$ and $(2l)^{\alpha} +2^{\alpha}\geq 6^{\alpha}+2^{\alpha}\geq 4^{\alpha}+4^{\alpha}$. Hence, (\ref{e22e}) follows directly as $\beta> 1$.

Finally, it is easy to see that (\ref{e23e}) holds when $\alpha\geq 1$ and $\beta>1$  by the monotonicity of $\left(s^\alpha+t^\alpha\right)^\beta$. Therefore, $\left(s^\alpha+t^\alpha\right)^\beta$ is special  escalating.
\end{proof}

 A $k$-polygonal cactus $G$ is called a {\em  cactus chain} if each polygon in $G$ has at most two cut-vertices and each cut-vertex is the common vertex of exactly two polygons. It is clear that each  cactus chain has exactly $n-2$ non-pendent polygons and two pendent polygons for $n\geq 2$. We denote by $\mathcal{A}_{n,k}$ the class consisting of those cactus chains such that each pair of  cut-vertices that  lies in the same polygon of $G$ are adjacent. In contrast, we denote by $\mathcal{B}_{n,k}$ the class consisting of those cactus chains such that each pair of  cut-vertices that  lies in the same polygon of $G$ are not adjacent. It can be seen that $\mathcal{A}_{n,k}$ is unique for $k\geq 3$ and  $\mathcal{B}_{n,3}=\emptyset$.

\begin{theorem}\label{15.1t} \cite{YJC2} Let $f(s,t)$ be a special escalating function and $G$ be a cactus of $\mathcal{G}_{n,k}$, where $n\geq3$ and $k\geq3$.\\
(i). If $k=3$, then
   $$I_f(G)\geq 2f(2,2)+2n f(4,2)+(n-2)f(4,4)$$
with equality holding if and only if $G\in \mathcal{A}_{n,3}$.\\
(ii). If $k\geq 4$, then
   $$I_f(G)\geq (kn-4n+4)f(2,2)+(4n-4)f(4,2),$$
where the equality holds if $G\in\mathcal{B}_{n,k}$. Furthermore, if $k\in \{4,5\}$, then the equality holds if and only if $G\in\mathcal{B}_{n,k}.$
\end{theorem}
\begin{corollary}\label{15.1c}  Let  $n\geq3,k\geq3,\alpha\geq 1,\beta>1$ and $G\in \mathcal{G}_{n,k}$.\\
(i). If $k=3$, then $SO_\alpha(G;\beta)\geq 2\big(2^{\alpha+1}\big)^{\beta}+
2n\big(4^{\alpha}+2^{\alpha}\big)^{\beta}+(n-2)\big(2\cdot4^{\alpha}\big)^{\beta}$, where the equality holds if and only if
$G\in \mathcal{A}_{n,3}$.\\
(ii). If $k\geq 4$, then $SO_\alpha(G;\beta)\geq (kn-4n+4)\big(2^{\alpha+1}\big)^{\beta}+(4n-4)\big(4^{\alpha}+2^{\alpha}\big)^{\beta}$, where the equality holds if $G\in\mathcal{B}_{n,k}$. Furthermore, if $k\in \{4,5\}$, then  the equality holds if and only if $G\in\mathcal{B}_{n,k}.$
\end{corollary}
\begin{proof} In Theorem \ref{15.1t}, set $f(s,t)=r_\alpha(s,t;\beta)$. Then the
corollary  follows immediately  by  Lemma \ref{15.3l} and a simple calculation.
\end{proof}

 \section{Acknowledgement}
 This work was supported by the National Natural Science Foundation of China [Grant number: 11971406].


\begin{thebibliography}{99}
\bibitem{Ali} A. Ali, D. Dimitrov, On the extremal graphs with respect to bond incident degree indices, {\em Discrete Appl. Math.}  238 (2018) 32-40.

%\bibitem{Bol} B. Bollob\'as, P. Erd\H{o}s, Graphs of extremal weights, {\em Ars Combin.}  50 (1998)  225-233.

\bibitem{Cru} R. Cruz, I. Gutman, J. Rada, Sombor index of chemical graphs,
    {\em Appl. Math. Comput.}  399 (2021) 126018.

\bibitem{Das}  K.C. Das, A.S. \c{C}evik, I.N. Cangul, Y. Shang, On Sombor index, {\em Symmetry} 13, 140 (2021), https://doi.org/10.3390/sym13010140.

\bibitem{De10} E. Deutsch, S. Klav\v zar, $M$-polynomial revisited: Bethe cacti and an extension of Gutman's approach, {\em J. Appl. Math. Comput.} 60 (2019) 253-264.

%\bibitem{Dos} T. Do\v sli\'c, T. R\'eti, D. Vuki\v cevi\'c, On the vertex degree indices of connected graphs, {\em Chem. Phys. Lett.}  512 (2011)  283-286.

\bibitem{Faj} S. Fajtlowicz, On conjectures of Graffiti-II, {\em Congr. Numer.} 60 (1987)  187-197.

\bibitem{Fur} B. Furtula, I. Gutman,   A forgotten topological index, {\em J. Math. Chem.} 53 (2015) 1184-1190.

\bibitem{Gu2013} I. Gutman, Degree-based topological indices, {\em Croat. Chem. Acta}  86  (2013) 351-361.

\bibitem{Gutman} I. Gutman, J. To\v sovi\' c,   Testing the quality of molecular structure descriptors. Vertex-degree-based topological indices, {\em J. Serb. Chem. Soc.} 78 (2013) 805-810.

\bibitem{Gu1972} I. Gutman, N. Trinajsti\'c, Graph theory and molecular orbitals. Total $\pi$-electron energy of alternant hydrocarbons, {\em Chem. Phys. Lett.}     17  (1972)  535-538.

\bibitem{Gu2020} I. Gutman, Geometric approach to degree-based topological indices: Sombor indices, {\em  MATCH
Commun. Math. Comput. Chem. } 86 (2021) 11-16.

\bibitem{Her} J.C. Hern\'andez, J.M. Rodr\'iguez, O. Rosario, J.M. Sigarreta,  Optimal inequalities and extremal problems on the general Sombor index, paper submitted arXiv:2108.05224.

\bibitem{Kulli2019} V.R. Kulli, The $(a,b)$-$KA$ indices of polycyclic aromatic hydrocarbons
and benzenoid systems, {\em International Journal of Mathematics Trends
and Technology}  65 (2019) 115-120.

\bibitem{Kulli2021} V.R. Kulli, I. Gutman, Computation of Sombor indices of certain
networks, {\em SSRG Int. J. Appl. Chem.} 8 (2021)  1-5.

\bibitem{Li} X. Li, J. Zheng, A unified approach to the extremal trees for different indices, {\em MATCH Commun. Math. Comput. Chem.}  54 (2005) 195-208.

\bibitem{Lin} Z. Lin, T. Zhou, V.R. Kullic, L. Miao, On the first Banhatti-Sombor index, preprint arXiv:2104.03615.

\bibitem{Mar1979} A.W. Marshall , I. Olkin,    {\it Inequalities: Theory of Majorization and Its Applications},  Academic Press,  New York,  1979.

\bibitem{Nik}  S. Nikoli\'c, G. Kova\v cevi\'c, A. Mili\v cevi\'c, N. Trinajsti\'c, The Zagreb indices 30 years after, {\em Croat. Chem. Acta} 76 (2003) 113-124.

\bibitem{Ran1} M. Randi\'c, On characterization of molecular branching, {\em J. Amer. Chem. Soc.} 97 (1975)  6609-6615.

\bibitem{Red} I. Red\v{z}povi\'{c},  Chemical applicability of Sombor indices, {\em J. Serb. Chem. Soc.}  86 (2021) 445-457.

\bibitem{Reti} T. R\'eti, T. Do\v sli\'c, A. Ali, On the Sombor index of graphs, {\em Contrib.
Math.} 3 (2021) 11-18.

\bibitem{Vu} D. Vuki\v cevi\'c, J. Durdevi\'c, Bond additive modeling 10.
    Upper and lower bounds of bond incident degree indices of catacondensed
    uoranthenes, {\em Chem. Phys. Lett.} 515 (2011) 186-189.

\bibitem{Wang2014} H. Wang, Functions on adjacent vertex degrees of trees with given degree sequence, {\em Cent. Eur. J. Math.\/}  12  (2014)  1656-1663.

\bibitem{YJC2} J. Ye, M. Liu, Y. Yao, K.C. Das, Extremal  polygonal cacti for  bond incident degree indices,
    {\em Discrete Appl. Math.}  257 (2019)  289-298.

%\bibitem{YJC1} J. Ye, Y. Yao, A note on the zeroth-order general Randi\'c index of polygonal cacti,
    %{\em Open J. Discrete Appl. Math.}  1 (2018)  01-07.

\bibitem{Zhou1} B. Zhou, N. Trinajsti\'c, On general sum-connectivity index, {\em J. Math. Chem.}   47  (2010)  210-218.
\end{thebibliography}
\end{document}